\documentclass[12pt]{amsart}
\usepackage{enumerate,amssymb,amsfonts,latexsym}
\usepackage{amscd,amsmath}
\usepackage{graphicx}
\usepackage{mathtools}
\usepackage{enumitem}

\newtheorem{thm}{Theorem}[section]
\newtheorem*{mainthm}{Main Theorem}

\newtheorem{lem}[thm]{Lemma}
\newtheorem{prop}[thm]{Proposition}
\theoremstyle{remark}
\newtheorem{rem}{Remark}[section]

\theoremstyle{definition}

\numberwithin{equation}{section}
\numberwithin{figure}{section}

\font\nt=cmr7

\def\note#1
{\marginpar
{\nt $\leftarrow$
\par
\hfuzz=20pt \hbadness=9000 \hyphenpenalty=-100 \exhyphenpenalty=-100
\pretolerance=-1 \tolerance=9999 \doublehyphendemerits=-100000
\finalhyphendemerits=-100000 \baselineskip=6pt
#1}\hfuzz=1pt}

\newcommand{\dist}{\operatorname{dist}}

\newcommand{\supp}{\operatorname{supp}}
\newcommand{\id}{\operatorname{id}}

\newcommand{\transverse}
 {\kern .7em\makebox[0pt][c]{\raisebox{.2ex}{$|$}}\kern -.6em\cap}

\newcommand{\tangent}
 {\kern .7em\makebox[0pt][c]{\raisebox{.77ex}{$--$}}\kern -.6em\cap}

\newcommand{\OO}{{\mathcal O}}

\def\BPhi{{\boldsymbol{\BPhi}}}
\def\B0{{\mathbf{0}}}

\newcommand{\Dom}{\operatorname{Dom}}

\catcode`\@=12

\def\Empty{}
\newcommand\oplabel[1]{
  \def\OpArg{#1} \ifx \OpArg\Empty {} \else
  	\label{#1}
  \fi}
		
\newcommand{\comm}[1]{}
\newcommand{\comment}[1]{}

\newcommand{\abs}[1]{\lvert#1\rvert}

\DeclarePairedDelimiter{\floor}{\lfloor}{\rfloor}
\newcommand\opR{\mathcal{R}}

\begin{document}

\title[Physical Measures ]{Physical Measures for Infinitely Renormalizable Lorenz Maps }
\author{M. Martens, B. Winckler}



\date{\today}

\begin{abstract} A physical measure on the attractor of a system describes the statistical behavior of typical orbits. An  example occurs in unimodal dynamics. Namely,  all infinitely renormalizable unimodal maps  have a physical measure. For Lorenz dynamics, even in the simple case of infinitely renormalizable systems, the existence of physical measures is more delicate. In this article we construct examples of infinitely renormalizable Lorenz maps which do not have a physical measure. A priori bounds on the geometry play a crucial role in (unimodal) dynamics. There are infinitely renormalizable Lorenz maps which do not have a priori bounds. This phenomenon is related to the position of the critical point of the consecutive renormalizations. The crucial technical ingredient used to obtain these examples without a physical measure, is the control of the position of these critical points.
\end{abstract}

\maketitle

\setcounter{tocdepth}{1}
\tableofcontents

\section{Introduction}\label{intro}

Sometimes the study of a flow can be reduced to the study of a one-dimensional map. The famous examples are the Lorenz maps. These interval maps are used to understand the dynamics of the flow introduced by E.~N.\ Lorenz \cite{Lo}. Indeed, this reduction is not always straightforward but in the case of the Lorenz flow it was shown to be valid \cite{T}. There is an extensive literature on Lorenz dynamics. As a brief introduction see for example \cite{V} or \cite{W}  and the references therein.

Although systems like Lorenz maps have a wide range of applications,
their theory is not well developed. As for unimodal maps, our topological understanding of Lorenz maps is complete. Unfortunately, most of the tools used to develop the geometric theory for unimodal maps could not be applied in the context of Lorenz maps. There are intrinsic obstructions to do so.

The main obstruction is the critical exponent of Lorenz maps. A Lorenz map is characterized by having a discontinuity at the critical point.  Near the critical point, the derivative tends to zero according to a power law. The exponent of this law is related to eigenvalues of a singularity in the original flow and can have any value,  not necessarily an integer. A crucial moment in the development of the unimodal theory was to consider unimodal maps which have holomorphic extensions. The same starting point can not be used when studying Lorenz maps because one has to consider critical exponents which are not necessarily an integer.

A crucial tool in unimodal and Lorenz dynamics is renormalization.  It was introduced to give  a very precise geometrical understanding of the dynamics \cite{CT,F}. Renormalization has also been  used to give a topological description of the dynamics.
A central part of any renormalization theory are the a priori bounds which state that the consecutive renormalizations form a precompact sequence of systems. Here arises a problem in Lorenz dynamics not present in the unimodal context. The renormalizations of Lorenz maps are also Lorenz maps. The critical point of these renormalizations might tend to the boundary of the domain of the system. There are examples where this phenomenon occurs. In such a case one does not have a priori bounds. The control of the position of the critical point of the renormalizations is the main difficulty which one encounters when studying Lorenz dynamics.

The third difficulty with Lorenz dynamics is the smoothness. Lorenz maps are obtained from invariant foliations associated with the original flow. Although the flow is smooth the invariant foliations generally have a very low degree of smoothness and hence the same may hold for Lorenz maps which occur in applications. The renormalization theory for unimodal maps applies to $\mathcal{C}^3$ maps. Indeed, unimodal map which are only $\mathcal{C}^2$ can have very uncontrolled geometry \cite{CMMT}. We will not address this problem.

The main reason to study Lorenz dynamics is their relevance in the broader study of flows. Another reason is the need to develop techniques which are able to deal with the specific challenges of Lorenz dynamics, techniques which might be applied beyond one-dimensional dynamics. The main result presented here is an example of a Lorenz map which displays ergodic behavior not present in unimodal dynamics. This is one example of a phenomenon which occurs in Lorenz dynamics not but not in smooth one-dimensional dynamics. Indeed, this is an invitation to look for other Lorenz surprises.

The simplest non-trivial dynamics occurs when the system is infinitely renormalizable. In the unimodal case such infinitely renormalizable systems are ergodic and have a minimal Cantor attractor. Almost every orbit converges to this invariant Cantor set. The Lebesgue measure of the attractor is zero \cite{M1}. Moreover, the Cantor set carries a unique invariant measure. It is the so-called physical measure. Asymptotically, typical orbits are distributed according to this measure.

\begin{mainthm}
Every monotone Lorenz family contains infinitely renormalizable maps which do not have a physical measure. These ergodic maps do have a minimal Cantor attractor of Lebesgue measure zero.
\end{mainthm}

Indeed, there are unimodal maps without physical measures, compare \cite{J,HK}. However, the dynamics of these maps is not at all as simple as the dynamics of infinitely renormalizable maps.

The general ingredients needed to construct these examples are collected in section \S \ref{prelim}. In particular, the construction of invariant measures on minimal Cantor sets is discussed. The examples presented in the Main Theorem have minimal Cantor sets which are clearly not uniquely ergodic,  they carry exactly two invariant measures.

The position of the critical point of the renormalizations is discussed in section \S\ref{critpts}. The combinatorial type of the renormalizations is unbounded. Without special care unbounded renormalization types might lead to unbounded geometry, no a priori bounds. In this section the estimates are prepared to find the delicate balance between unbounded renormalization type, needed to obtain systems without physical measures, and controlled position of the critical points, needed to control the geometrical properties of the system.

The actual examples are constructed in section \S \ref{construc}.
The idea is that given that we have two ergodic measures on the attractor we can force many points to spend a long time following one measure.  We then force them to wander over to the other measure and spend an even longer time following it.  The process is then repeated.

\bigskip
\noindent
{\bf Acknowledgement.} The examples were constructed during visits of the authors to KTH, Institut Mittag-Leffler and Stony Brook University. The authors would like to thank these institutes for their kind hospitality.

\newcommand\dfn[1]{\emph{#1}}

\section{Invariant Measures and Preliminaries}\label{prelim}

\subsection{Lorenz maps}

The \dfn{standard Lorenz map} $q: [0,1]\setminus\{c\} \to [0,1]$
with \dfn{critical point} $c \in (0,1)$,  \dfn{critical exponent} $\alpha > 1$,
and \dfn{critical values} $u, 1-v\in [0,1]$, is defined by
\begin{equation} \label{stdfamily}
    q(x) = \begin{cases}
        u\cdot\left( 1 - \left( \frac{c-x}{c} \right)^\alpha \right),
        & x \in [0,c), \\
        1 + v\cdot\left( -1 + \left( \frac{x-c}{1-c} \right)^\alpha \right),
        & x \in (c,1].
    \end{cases}
\end{equation}

A \dfn{Lorenz map} on $[0,1]$ is any map $f:[0,1]\setminus\{c\} \to [0,1]$ of
the form
\[
    f(x) = \begin{cases}
        \phi_- \circ q(x), & x \in [0,c), \\
        \phi_+ \circ q(x), & x \in (c,1], \\
    \end{cases}
\]
where $\phi_-$ and $\phi_+$ are increasing diffeomorphisms on $[0,1]$ which we
call the \dfn{diffeomorphic parts} of~$f$.  See Figure~\ref{fig:lorenzmap} for
an illustration of a Lorenz map.

The critical point is always denoted by $c$, or $\text{crit}(f)$ if we wish to
emphasize the map we are talking about.  The branches of $f$ are denoted
$f_-:[0,c]\to[0,1]$ and $f_+:[c,1]\to[0,1]$, where we define $f_-(c) =
\lim_{x\uparrow c} f(x)$ and $f_+(c) = \lim_{x\downarrow c} f(x)$.  Note that
$f_-(c)$ and $f_+(c)$ are the critical values of~$f$.  We say that $f_-$ is a
\dfn{full branch} if it is onto and we say that $f_-$ is a \dfn{trivial branch}
if $f_-([0,c])=[0,c]$.  Similarly, $f_+$ is full if it is onto and $f_+$ is
trivial if $f_+([c,1])=[c,1]$.  A Lorenz map is said to be \dfn{nontrivial} if
the images of both branches contain the critical point in their interior, i.e.\ if $f_+(c) < c <
f_-(c)$, and it is said to be \dfn{full} if both branches are full, i.e.\ if
$f_-(c)=1$ and $f_+(c)=0$.

We make the following assumptions on all Lorenz maps $f$ throughout this
article unless otherwise stated:
\begin{enumerate}
    \item the critical exponent $\alpha>1$ is fixed once and for all (and we
        allow noninteger $\alpha$),
    \item the diffeomorphic parts of $f$ are
        $\mathcal{C}^3$--diffeomorphisms with negative Schwarzian derivate,
    \item $f$ is nontrivial,
    \item the only fixed points of $f$ are $0$ and~$1$.\footnote{If there were
        other fixed points we could rescale $f$ to the smallest interval
        containing $c$ and exactly two fixed points.}
\end{enumerate}

\subsection{Families of maps} \label{families}

A \dfn{family of Lorenz maps} is a set of Lorenz maps parametrized by a compact
and simply connected subset $U \subset \mathbb{R}^2$.  The map $U \ni \lambda
\mapsto F_\lambda$ is required to be at least continuous, but in some places we
will need more smoothness.
We say that $U \ni \lambda \mapsto F_\lambda$ is a \dfn{full family} if it
realizes all possible combinatorics (see~\cite{MM}).  We shall at times need
families to satisfy one or more of the following properties:
\begin{enumerate}[label=(F\arabic*)]
    \item $(\lambda_1,\lambda_2) \mapsto F_\lambda(x)$ is a monotone function
        in $\lambda_1$ and in $\lambda_2$, for every $x$, \label{F1}
    \item there exists a unique $\hat\lambda \in U$, called the
        \dfn{full vertex}, such that both branches of $F_{\hat\lambda}$ are
        full (see Figure~\ref{fig:lorenzmap}), \label{F2}
    \item there exists a unique $\check\lambda \in U$, called the
        \dfn{trivial vertex}, such that both branches of $F_{\check\lambda}$
        are trivial, \label{F3}
    \item the fixed points of $F_\lambda$ (i.e.\ $0$ and $1$) are hyperbolic
        repellers for all $\lambda \in U$. \label{F4}
\end{enumerate}
Note that we do allow maps with trivial branches on the boundary of~$U$ even
though such maps are not nontrivial.  All other maps in a family are required
to be nontrivial.

A family which satisfies all of the properties \ref{F1}--\ref{F4} is called a
\dfn{monotone family}.  The following theorem is taken from \cite{MM}:

\begin{thm} \label{monotonefull}
    Every monotone family is a full family.
\end{thm}

Given a fixed $c \in (0,1)$ we define the \dfn{standard Lorenz family} $Q$ by
$(u,v) \mapsto Q_{u,v,c}$, where $Q_{u,v,c}=q$ is given by \eqref{stdfamily}.
Note that the standard family is a monotone family.  All other
families $F$ we consider are \dfn{close to standard maps}, by which we
mean that the diffeomorphic parts of~$F_\lambda$ are close to identity in the
$C^2$--norm, for every $\lambda \in U$. If the diffeomorphic parts of a Lorenz
map $f$ have $C^2$--distance to identity bounded by $\epsilon>0$, we say that
$f$ is {\it $\epsilon$--close to standard maps}.

The standard family is a bit
special in that the critical point does not vary across the family, whereas for
general families the critical point depends on $\lambda \in U$.
Note that we in general have no control over the behavior of the critical point
in terms of~$\lambda$ for the families we study.  To overcome this problem we
need to assume that our families are analytic  (e.g.\ in
Lemma~\ref{fam depend}).

\subsection{Renormalization}

A Lorenz map $f$ is \dfn{renormalizable} if there exists a closed interval
$C \subsetneq [0,1]$ such that the first-return map to~$C$ is
a nontrivial Lorenz map on~$C$.  The \dfn{renormalization operator}
$\opR$ is defined by taking the largest such $C$ which properly contains the
critical point and sending $f$ to its first-return map on $C$, affinely
rescaled to~$[0,1]$.  We call $\opR f$ the \dfn{renormalization} of~$f$.

Let $f$ be renormalizable with return interval $C \ni c$ and consider the
orbits of $C^- = C \cap [0,c)$ and $C^+ = C \cap (c,1]$.  By definition there
exist minimal $a,b\geq1$ such that $f^{a+1}(C^-)$ and $f^{b+1}(C^+)$ are contained in
$C$.  Since the first-return map is again a Lorenz map (on~$C$) it follows that
the left and right boundary points of~$C$ are periodic points (of period $a+1$
and~$b+1$, respectively), and since $\opR f$ is nontrivial
$C^- \subset f^{a+1}(C^-) \subset C$ and
$C^+ \subset f^{b+1}(C^+) \subset C$.
If the forward orbits of $f(C^-)$ and $f(C^+)$ stay to the right and to the
left of the critical point respectively before returning to~$C$, then $f$ is
said to be of \dfn{monotone type} and we say that $f$ is
\dfn{$(a,b)$--renormalizable}, see Figure~\ref{fig:lorenzmap}.  Throughout the
rest of this article we will only consider renormalizations of monotone types
(there are other types but they are more difficult to analyze).

\begin{figure}
  \begin{center}
    \includegraphics{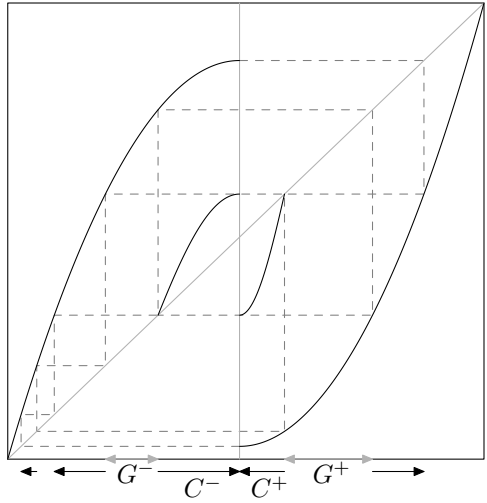}
  \end{center}
  \caption[A Lorenz map]{A Lorenz map and its first-return map to the return
      interval $C = C^- \cup C^+$.  The
      map is renormalizable since the first-return map is
      a Lorenz map on~$C$.  In fact, it is
      $(1,2)$--renormalizable since $C^-$ maps
      to the right of the critical point then returns, and
      $C^+$ maps to the left of the critical point, then remains on
      the left for one more step before it returns.  Furthermore, it
      corresponds to a
      full vertex in some family since the first-return map is onto on~$C$.}
      \label{fig:lorenzmap}
\end{figure}

We say that $f$ is \dfn{infinitely renormalizable} if it can be repeatedly
renormalized infinitely many times.  An infinitely renormalizable map is said
to be of \dfn{combinatorial type} $\{(a_n,b_n)\}_{n=1}^\infty$ if $\opR^{n-1}
f$ is $(a_n,b_n)$--renormalizable, for all $n\geq1$.

Let $U\ni\lambda \mapsto F_\lambda$ be a family of Lorenz maps.  The set of
$\lambda\in U$ such that $F_\lambda$ is $(a,b)$--renormalizable is called the
\dfn{$(a,b)$--archipelago} of~$F$ and we denote it by $A_{a,b}$.  An
\dfn{$(a,b)$--island} is defined as the closure of a connected component of the
interior of~$A_{a,b}$.  We will talk about ``islands'' and ``archipelagos''
when the return times $(a,b)$ are irrelevant.  We say that $D_{a,b} \subset
A_{a,b}$ is a \dfn{full island} if $D_{a,b}$ is an island and
$D_{a,b}\ni\lambda \mapsto \opR F_\lambda$ is a full family.  The following
theorem is taken from \cite{MM}:

\begin{thm}
    Every archipelago of a monotone family contains a full island.
\end{thm}

\subsection{Covers}

Let $f$ be an infinitely renormalizable map.  There exists a nested sequence
of intervals $C_1 \supset C_2 \supset \cdots$ (all containing the
critical point) on which the corresponding first-return map is again a
Lorenz map.
The critical point splits each $C_n$ into two subintervals which we denote
$C_n^- = C_n \cap [0,c)$ and $C_n^+ = C_n \cap (c,1]$.
Let $T_n^-$ and $T_n^+$ denote the first-return times of $C_n^-$ and $C_n^+$
to $C_n$, respectively.
The \dfn{$n^\text{th}$ level cycles} $\Lambda_n^-$ and $\Lambda_n^+$ of~$f$ are
the following collections of closed intervals
\begin{equation}\label{cycle}
    \Lambda_n^- = 
        \{ \overline{f^k(C_n^-)} : 0 \leq k \leq T_n^- - 1 \},
\end{equation}
and similarly for $\Lambda_n^+$.
Let $\Lambda_0^-=[0,c]$, $\Lambda_0^+=[c,1]$ and let
$\Lambda_n = \Lambda_n^- \cup \Lambda_n^+$, for $n\geq0$.  The interiors of
elements in $\Lambda_n$ are pairwise disjoint (see~\cite{W}) and
cycles are nested in the sense that if $I\in\Lambda_{n+1}$, then
$I\subset J$ for a unique $J\in\Lambda_n$.  The intersection of all levels is
denoted
\begin{equation} \label{attractor}
    \OO_f = \bigcap_{n\geq0} \bigcup\Lambda_n.
\end{equation}

The proof of the following lemma can be found in \cite{MW,W}, (see also
\cite{M1}).  In the statement $G^-_n$ denotes the connected component of $C_{n-1}\setminus \Lambda_n$ adjacent at the left to $C_n$ and $G^+_n$ denotes the connected component adjacent to the right of $C_n$. They are called the {\it gaps} to the left and right of $C_n$, see Figure~\ref{fig:lorenzmap}.

\begin{lem}\label{ergnowand}
If the Lorenz map $f$ is infinitely renormalizable with
$$
\liminf_{n\to\infty} \frac{\abs{C_n}}{\abs{G^\pm_n}}\le \epsilon,
$$
then
\begin{itemize}
\item $f$ is ergodic,
\item  $\OO_f$ is the attractor, it is the limit set of almost every point,
\item $\OO_f$ has Lebesgue measure $0$.
\end{itemize}
\end{lem}

Let $f$ be an $n$ times renormalizable map. For such a map we will consider the collection of intervals
$$
\mathcal{C}_n=\{I\subset [0,1]\mid \exists e_I\ge 0 \text{ with }
f^{e_I}:I \to C_n \text{ is monotone and onto }\}.
$$
Observe, the numbers $e_I$ are unique. The proof of the following lemma 
can be found in \cite{W,MW} (see also \cite{M1}).

\begin{lem}\label{C} The collection $\mathcal{C}_n$ consists of pairwise disjoint intervals, with $\abs{\bigcup \mathcal{C}_n}=1$. Moreover, $\mathcal{C}_{n+1}$ is a refinement of $\mathcal{C}_n$.
\end{lem}

To each interval $I\in \mathcal{C}_n\setminus\{C_n\}$,  we can assign a word
$$
\omega_I=(\omega_I(0), \omega_I(1),\omega_I(2),\cdots, \omega_I(e_I-1))\in \{L,R\}^{e_I}
$$
with
$\omega_I(k)=L$ if $f^k(I)\subset [0,c)$ and $\omega_I(k)=R$ if $f^k(I)\subset (c,1]$.
The following lemma has a straight forward proof. It serves mainly as a definition of {\it corresponding collection}.

\begin{lem}\label{correspond} Let $\epsilon>0$ and a finite set $X(f)\subset \mathcal{C}_n(f)$ be given. For $\tilde{f}$ close enough to $f$ there is a set
$X(\tilde{f})\subset \mathcal{C}_n(\tilde{f})$ and an identification
$h:X(f)\to X(\tilde{f})$ such that
$$
\omega_{h(I)}=\omega_I.
$$
Moreover, $|X(\tilde{f})|\ge |X(f)|-\epsilon$.
The collection $X(\tilde{f})$ is called the collection corresponding to $X(f)$.
\end{lem}

\subsection{Measures}

Let $I$ be an interval, let $g: I\to I$, and let $\delta_x$ denote the
Dirac measure at~$x$. A measure $\mu: I \to \mathbb{R}$ is called a
\dfn{physical measure} (for~$g$) if
\[
    \frac 1 n \sum_{k=1}^n \delta_{g^n(x)} \to \mu,
\]
in the weak-$\star$ topology, for (Lebesgue-)almost every $x\in I$.

We will now describe the construction of invariant measures on the attractor of
infinitely renormalizable maps.
Let $f$ be an infinitely renormalizable map of combinatorial type
$\{(a_n,b_n)\}_{n=1}^\infty$ and to avoid technicalities we assume that $f$ has no wandering
intervals.
Let $\Sigma_n(f)$ denote the $\sigma$--algebra generated by
$\bigcup_{k=0}^n \Lambda_k(f)$, i.e.\ all cycles up to and including level $n$.
Define the \dfn{$n^\text{th}$ level basis measures}
$\nu_n^-,\nu_n^+: \Sigma_n \to \mathbb{R}$ by
\[
    \supp\nu_n^- = \bigcup\Lambda_n^- \quad\text{and}\quad \nu_n^-(I) = 1,
    \text{ for } I \in \Lambda_n^-,
\]
and similarly for $\nu_n^+$.

Since cycles are nested there are well defined maps
$\pi_n: \Lambda_{n+1} \to \Lambda_n$, defined by $\pi_n(I) = J$ where
$J\in\Lambda_n$ is the interval which contains~$I\in\Lambda_{n+1}$.
Let $H_1(\Lambda_n)$ be the measure space spanned by the $n^\text{th}$ level basis
measures, $H_1(\Lambda_n) = \{ x \nu_n^- + y \nu_n^+ : x,y\in\mathbb{R}\}$.
We use the convention that $\nu_n^-$ corresponds to the first coordinate and
$\nu_n^+$ corresponds to the second coordinate.
Following the arguments of \cite{GM} we find that the push-forward
by $\pi_n$ takes $H_1(\Lambda_{n+1})$ to $H_1(\Lambda_n)$ and that the
representation of this map in the basis measures on levels $n+1$ and~$n$
is given by the \dfn{winding matrix}
\begin{equation} \label{winding}
    W_n = \begin{pmatrix} 1 & b_{n+1} \\ a_{n+1} & 1 \end{pmatrix}.
\end{equation}
Let  $M \subset \mathbb{R}^2$ be the positive quadrant.
We have the following identification (see \cite{GM})
\begin{prop}
The invariant measures on $\OO_f$ are isomorphic to the inverse limit 
\[
    M \xleftarrow{W_0} M \xleftarrow{W_1} M \xleftarrow{W_2} \dotsm.
\]
\end{prop}

Note that the transposed winding matrix can be used to calculate the
first-return times of $C_n^\pm$ according to
$$
T_{n+1} = W_n^t T_n,
$$
where $T_k = (T_k^-,T_k^+)$ and $T_0=(1,1)$.

\begin{rem}  The idea that  a measure plays the role of a cycle in homology and the times  play the role of  a cocycle in cohomology
was explored in \cite{GM} for general minimal Cantor sets.
\end{rem} 

 Normalize the $n^\text{th}$ level basis measures to obtain the 
 \dfn{$n^\text{th}$ level probability basis measures} 
 $$
 \mu^\pm_n=\frac{\nu^\pm_n}{T^\pm_n}.
 $$

  \begin{lem}\label{ergodicmeasure} For every $\epsilon>0$ there exists a sequence $K_n\to \infty$ such that for every sequence  $\{(a_n, b_n)\}_{n=1}^\infty$ with $a_n, b_n\ge K_n$, the map on the corresponding Cantor set  $\OO_f$ has exactly two ergodic probability measures, $\mu^+$ and $\mu^-$. Moreover,
 $$
 \mu^+([0,c])\le \epsilon, \text{ and }  \mu^+([c,1])\ge 1- \epsilon,
 $$ 
 and
  $$
 \mu^-([0,c])\ge 1- \epsilon, \text{ and }  \mu^-([c,1])\le \epsilon.
 $$ 
 The $n^\text{th}$ level probability basis measures $\mu^\pm_n$, $n\ge 2$ even, satisfy the same estimates. For $n>2$ odd we have
 $$
 \mu^-_n([0,c])\le \epsilon, \text{ and }  \mu^-_n([c,1])\ge 1- \epsilon,
 $$ 
 and
  $$
 \mu^+_n([0,c])\ge 1- \epsilon, \text{ and }  \mu^+_n([c,1])\le \epsilon.
 $$ 
 \end{lem}
 
 \begin{proof} Let $\mathbb{P}_n\subset \mathbb{R}^2$ be the convex hull of $\mu^+_n$ and $\mu^-_n$. The map $p_n:\mathbb{P}_{n+1}\to \mathbb{P}_n$ is given by
\begin{equation}\label{mumin}
\mu^-_{n+1}\mapsto  \frac{a_{n+1} T^+_n}{a_{n+1} T^+_n+ T^-_n} \mu^+_n+
\frac{T^-_n}{a_{n+1} T^+_n+ T^-_n} \mu^-_n. 
\end{equation}
and
\begin{equation}\label{muplus}
\mu^+_{n+1}\mapsto \frac{T^+_n}{T^+_n+b_{n+1} T^-_n} \mu^+_n+
 \frac{b_{n+1} T^-_n}{T^+_n+b_{n+1} T^-_n} \mu^-_n, 
\end{equation}

Let $p^0_n:\mathbb{P}_n\to \mathbb{P}_0$ be the projection.

The proof of the lemma will be by induction.  Observe, $\mu^+_0([c,1])=1$, $\mu^+_0([0,c])=0$,
 $\mu^-_0([c,1])=0$, and $\mu^-_0([0,c])=1$. This means that the lemma holds for $n=0$. Suppose the lemma holds for $n\ge 0$.
 $$
\dist( p^0_n(\mu^\pm_n), \{\mu^+_0, \mu^-_0\})<\epsilon.
 $$
 Then, there is a $\delta_n>0$ such that if  $\mu\in \mathbb{P}_n$ with
  $$
\dist( \mu, \{\mu^+_n, \mu^-_n\})<\delta_n
 $$
 then
 $$
\dist( p^0_n(\mu), \{\mu^+_0, \mu^-_0\})<\epsilon.
 $$
The equations (\ref{muplus}) and (\ref{mumin}) imply that for $a_n, b_n\ge 1$ large enough we get
 $$
 \dist(p_n(\mu^\pm_{n+1}), \{\mu^+_n, \mu^-_n\})<\delta_n.
 $$
 The nature of the equations   (\ref{muplus}) and (\ref{mumin}) implies that for $n\ge 0$ even we have
 $$
 \dist( p^0_n(\mu^\pm_n), \mu^\pm_0)<\epsilon,
 $$
 and for odd $n\ge 0$ we have
$$
 \dist( p^0_n(\mu^\pm_n), \mu^\mp_0)<\epsilon.
 $$ 
  This finishes the induction and the lemma follows.
  \end{proof}

 \begin{lem}\label{measurebalance} There exists $\epsilon>0$
 with the following property.  For every $K>0$ there exist $w>0$ such that for every full Lorenz map $f$ which is  $\epsilon$--close to standard maps and $c=\text{crit}(f)\in [\frac{1}{K}, 1-\frac{1}{K}]$,
 $$
 \mu([0,c]) \ge w \text{ and } \mu([c,1])\ge w,
 $$
 where $\mu$ is the (unique) absolutely continuous invariant probability measure of $f$.
 \end{lem}
 
 \begin{proof} Let $\mathcal{B}_n$ be the collection of branches of $f^n$: if the interval $I\in \mathcal{B}_n$ then $f^I=f^n: I\to [0,1]$ is monotone and onto. In this situation there is an asymptotic expression for the invariant measure:
 \begin{equation}\label{PF}
 \lim_{n\to \infty} \frac{1}{n} \sum_{k=0}^{n-1} \sum _{I\in \mathcal{B}_k} f^I_*(\lambda|I) ([0,c])=\mu([0,c]),
 \end{equation}
 where $\lambda|I$ represents the restriction of the Lebesgue measure to the interval $I$. 
 
 For each $I\in \mathcal{B}_k$ let $I^-=(f^I)^{-1}([0,c])$. We will show that there is a $w>0$ such that 
 \begin{equation}\label{muuu}
     \abs{I^-}/\abs{I}\ge w,
 \end{equation}
 for every $I\in \mathcal{B}_k$. 
 
 To prove (\ref{muuu}) we decompose each branch $f^I$ in four parts. Namely, for every branch $I\in \mathcal{B}_k$ there exist
$t, t'\ge 0$, and an interval $I_1$, with $c\in \partial I_1$ and $I_1'=f(I_1)$, and  a second interval $I_2=[0,c] \text{ or } [c,1]$ such that
$$
f^I=f|I_2\circ f^{t'}|I'_1\circ f| I_1\circ f^t|I,
$$ 
where the surjective maps $f^t:I\to I_1$ and $f^{t'}:I_1'\to I_2$ have monotone extensions mapping onto $[0,1]$. These extensions are  branches in $\mathcal{B}_t$ and  $\mathcal{B}_{t'}$ resp. The condition on the map $f$, namely that it is close enough to a standard full map, and hence has derivative  definitely above $1$ and uniformly bounded second derivative in a large neighborhood of $0$ and $1$, imply that the branch $f^{t'}|I'_1$
 has uniformly bounded distortion (in terms of $\epsilon>0$). For this type of branch in $\mathcal{B}_{t'}$ the distortion is controlled by how fast the preimages are contracted towards the fixed points and a bound on the nonlinearity of the original map.

If the branch $f^{t}|I$ is of the same type, that is when $I_1=[0,c] \text{ or }[c,1]$ then the same reasoning gives a uniform bound on the distortion of the first part of the decomposition. Otherwise,  it is the Koebe Lemma, \cite{MS}, which gives a uniform bound on the distortion.

The single iterates $f|I_2$ and $f|I_1$ are also unable to cause that the preimage $(f^I)^{-1}(c)$, where $c\in [\frac{1}{K}, 1-\frac{1}{K}]$, is too close to the boundary of $I$ relative to the size of $I$.
The estimate (\ref{muuu}) follows.

 The limit (\ref{PF}) and estimate (\ref{muuu}) imply $\mu([0,c])\ge w$. Similarly, we can find a uniform estimate for $\mu([c,1])\ge w$.
 \end{proof}

\section{The Critical Point of the Renormalizations}\label{critpts}

In this section we will fix an analytic full family of Lorenz maps close enough to standard maps. It has a
unique full map $\hat{f}$ with critical point $\hat{c}$. For every
$a,b\in \mathbb{N}$ we will consider the archipelago $A_{a,b}$ of maps which
are $(a,b)$--renormalizable. Assume that each archipelago $A_{a,b}$ has a full
island $D_{a,b}$ which has a unique map whose renormalization is full.  This
map is called the full vertex of the island $D_{a,b}$. The constants in the
following statements will be dependent on the family. The main objective of
this section is to get a precise control over the critical point
of the renormalization of the full vertex.

\bigskip

\begin{lem}\label{dzetak} There exists $\rho<1$ such that for every $f\in D_{a,b}$ and $k\ge 1$
\begin{equation}\label{dzet+}
    \frac{f^{-k}_-(c)}{c} \le \rho^k,
\end{equation}
and
\begin{equation}\label{dzet-}
    \frac{1-f^{-k}_+(c)}{1-c}\le \rho^k.
\end{equation}
In particular, $f_+(c)\le \rho^{b-1}$ and  $1-f_-(c)\le \rho^{a-1}$.
\end{lem}

\begin{proof}
The last sentence follows from \eqref{dzet+} and \eqref{dzet-} since
$f_-(c) \geq f_+^{-a+1}(c)$ and $f_+(c) \leq f_-^{-b+1}(c)$ if $f$ is
$(a,b)$--renormalizable.  We now prove \eqref{dzet+}.  The proof of
\eqref{dzet-} follows by symmetry.

Consider the line
\[
    l(x) = \beta_f x,\qquad \beta_f = \min\{f'(0),f_-(c)/c\}.
\]
We claim that $f_-(x) \geq l(x)$, $\forall x \in [0,c]$.  Suppose not, so that
$f(x) < l(x)$ for some $x \in (0,c)$ since $f(0)=l(0)$ and $f_-(c) \geq l(c)$.
Without loss of generality we may assume
$f'(x) < l'(x) = \beta_f$.   Since $f_-(c) \geq l(c)$ and
$f_-'(c) = 0$, $\exists y \in (x,c)$ such that $f(y)=l(y)$ and
$f'(y) \geq \beta_f$.  By the minimum principle (see~\cite{MS})
$f'(x) \geq \min\{f'(0),f'(y)\} \geq \beta_f$ which contradicts $f'(x) <
\beta_f$.

From the above we get that $f_-^{-k}(c) \leq l^{-k}(c) = \beta_f^{-k} c$.
Let
\[
    \rho = \sup_{f \in \bigcup A_{a,b}} \beta_f^{-1}.
\]
We have to show that $\rho < 1$.  Note that $f'(0)$ is uniformly bounded away
from~$1$ as long as the family is close enough to standard maps
(since $q'(0) = \alpha u/c \geq \alpha > 1$ for a nontrivial standard map),
so we only need to find a uniform lower bound on $f_-(c)/c$.

Let $I = [f_-^{-1}(c),c]$ so that $f^2(I) = [f_+(c), f(f_-(c))]$.
Since $f$ is renormalizable this implies that
$I \subset f^2(I)$ and consequently $Df^2(x) > 1$ for some $x \in I$.
A calculation using the fact that the critical point and the derivative of the
diffeomorphic parts have uniform bounds over the family shows that
$Df^2(y) \to 0$ as $f_-(c)\to c$, $\forall y \in I$.  In particular,
$f_-(c)$ must be uniformly bounded away from $c$ if $f$ is renormalizable.
\end{proof}

\begin{lem}\label{fam depend} There exists $\rho<1$ such that for every $f\in D_{a,b}$
\begin{equation}\label{c2dist}
\dist_{C^2}(f, \hat{f})=O(\rho^{\min\{a,b\}}).
\end{equation}
In particular,
\begin{equation}\label{dfdist0}
\left\lvert\frac{Df(0)}{D\hat{f}(0)}-1\right\rvert=O(\rho^{\min\{a,b\}}),
\end{equation}
\begin{equation}\label{dfdist1}
\left\lvert\frac{Df(1)}{D\hat{f}(1)}-1\right\rvert=O(\rho^{\min\{a,b\}}).
\end{equation}
\end{lem}

\begin{proof} The family $F$ is defined on the parameter set $U$. Without loss of generality we may assume that the full map $\hat{f}$ of the family corresponds to the parameter values $\lambda_1=\lambda_2=0$. 
Consider the function $V:U\to \mathbb{R}^2$ defined by
$$
V(\lambda_1, \lambda_2)=(1-F(\lambda_1,\lambda_2)(c_-(\lambda_1, \lambda_2)), F(\lambda_1,\lambda_2)(c_+(\lambda_1, \lambda_2))).
$$
This function has a holomorphic extension to a neighborhood of $(0,0)\in \mathbb{C}^2$. Observe, $V(0,0)=(0,0)$ and, according to Lemma \ref{dzetak},  $$
V(D_{a,b})\subset [0, \rho^{\min\{a,b\}-1}]^2.
$$
The parameters $(0,0)$ correspond to the full map of the family. So, the image of $V$ contains a neighborhood of $(0,0)\in \mathbb{R}^2$. This implies that $V$ has in all directions some non-zero higher order derivative. 
Hence, for some $D\ge 1$ large enough,
$$
\dist(D_{a,b},0)=O(( \rho^{\min\{a,b\}})^\frac{1}{D}).
$$
The functions in the family depend  smoothly on the parameters, and the lemma follows.
\end{proof}

If a map $f$ is $(a,b)$--renormalizable $f^{-b}(c)$ denotes the corresponding preimage of the critical point with the branch $f_-$. Similarly, $f^{-a}(c)$ is the corresponding preimage with the branch $f_+$.

\begin{lem}\label{gamma} There exist $\gamma_-, \gamma_+>0$ such that for every $K>0$ and every sequence $a,b\to \infty$ with $\frac{1}{K}\le \frac{a}{b}\le K$ and $f_{a,b}\in D_{a,b}$
\begin{equation}\label{gamma-}
\lim_{a,b\to \infty} \frac{Df_{a,b}^b(f^{-b}_{a,b}(c_{a,b}))}{D\hat{f}^b(0)}=\gamma_->0,
\end{equation}
and
\begin{equation}\label{gamma+}
\lim_{a,b\to \infty} \frac{Df_{a,b}^a(f^{-a}_{a,b}(c_{a,b}))}{D\hat{f}^a(1)}=\gamma_+>0.
\end{equation}
\end{lem}

\newcommand{\clo}[1]{x_{#1}} 
\newcommand{\fclo}[1]{\hat x_{#1}} 

\begin{proof}
Let $f = f_{a,b}$, $\clo k = f_-^{-k}(c)$ and
$\fclo k = \hat f_-^{-k}(\hat c)$.  Then
\[
    \log \frac{Df^b(\clo b)}{D\hat f^b(0)}
    =
    b \log \frac{Df(0)}{D\hat f(0)}
        + \sum_{k=1}^b \log \frac{Df(\clo k)}{Df(0)}.
\]
Use \eqref{dfdist0} and the fact that $a/b$ is bounded from above and
below to see that the first term on the right-hand side tends to zero as
$a,b\to\infty$.  Let $\sigma(f)$ denote the second term on the right-hand side.
We claim that $\sigma(f_{a,b}) \to \log \gamma_-$ as $a,b\to\infty$, where
\[
    \gamma_-
    =
    \lim_{n\to\infty} \gamma_-(n)
    =
    \lim_{n\to\infty}
    \exp\left\{
        \sum_{k=1}^n \log \frac{D\hat f(\fclo k)}{D\hat f(0)}
        \right\}.
\]
Note that $\fclo k= O(D\hat{f}(0)^{-k})$, so the
limit exists and $\gamma_->0$. In particular,
\begin{equation}\label{gsum}
    \left\lvert \sum_{k\ge n} \log \frac{D\hat f(\fclo k)}{D\hat f(0)}
        \right\rvert
    = O(D\hat{f}(0)^{-n}).
\end{equation}
 We will now prove the claim which in turn
implies \eqref{gamma-}.  The proof of \eqref{gamma+} follows by symmetry.

Use Taylor expansion around $0$ with mean-value form of the remainder to see
that there exist $t_k \in [0,\clo k]$ such that
\[
    \log \frac{Df(\clo k)}{Df(0)}
    =
    \clo k \frac{D^2f(t_k)}{Df(t_k)}.
\]
The fraction on the right-hand side is the nonlinearity of $f$ at $t_k$ which
we denote $N_f(t_k)$.  Note that the nonlinearity is uniformly bounded away
from the critical point, i.e.\ there exists $K < \infty$ not depending on~$f$
such that for $k$ large enough
\[
    \abs{N_f(t_k)} \leq \max_{t \leq \rho^k} \abs{N_f(t)} \leq K.
\]
Now fix $B>0$.  The bound on the nonlinearity and \eqref{dzet+} show that
\begin{equation}\label{Nsum}
    \left\lvert\sum_{k=B+1}^{b} x_k N_f(t_k)\right\rvert=
        O(\rho^{B+1}).
\end{equation}
Lemma~\ref{fam depend} implies that
\begin{equation}\label{xconv}
\sum_{k=1}^{B} \log \frac{Df(\clo k)}{Df(0)}\to 
\sum_{k=1}^B  \log \frac{D\hat f(\fclo k)}{D\hat f(0)}
\end{equation}
as $a,b\to\infty$. Observe,
$$
\begin{aligned}
    \left\lvert\sigma(f_{a,b})-\log \gamma_-\right\rvert\le
&\left\lvert \sum_{k=1}^B \log \frac{Df(\clo k)}{Df(0)}-
\sum_{k=1}^B \log \frac{D\hat f(\fclo k)}{D\hat f(0)} \right\rvert+\\
& \left\lvert\sum_{k=B+1}^{b} x_k N_f(t_k)\right\rvert+
\left\lvert \sum_{k\ge B+1} \log \frac{D\hat f(\fclo k)}{D\hat f(0)}\right\rvert.
\end{aligned} 
$$
Now, use (\ref{gsum}), (\ref{Nsum}), and (\ref{xconv}) and 
the convergence $\sigma(f_{a,b}) \to \log\gamma_-$ follows. 
\end{proof}

Given a map $f\in D_{a,b}$, let $G^-$ and $G^+$ be the \emph{gaps} adjacent to the return interval $C$, see Figure~\ref{fig:lorenzmap}.

\begin{lem}\label{geobounds}  For every $\epsilon>0$ there exists $N\ge 1$ such that if $a, b\ge N$ then for every map  $f\in D_{a,b}$
\begin{equation}\label{gaps}
\max\left\{ \frac{\abs{C}}{\abs{G^-}}, \frac{\abs{C}}{\abs{G^+}} \right\} \le \epsilon
\end{equation}
and $\opR f$ is $\epsilon$--close to a standard map.
\end{lem}

\begin{proof}
The bound~\eqref{gaps} follows from Lemma~\ref{dzetak}.
To see this let $C^- = C \cap [0,c)$ and $C^+ = C \cap (c,1]$.  Then $f(C^-)
\subset [f_+^{-a+1}(c),1]$ and $f(C^+) \subset [0,f_-^{-b+1}(c)]$.  In
particular, both $\abs{C^\pm} \to 0$, uniformly as $a,b \to
\infty$, by Lemma~\ref{dzetak}.  It follows that $\abs{C} \to 0$ and
consequently $\abs{f_-^{-1}(C)} \to 0$, uniformly as $a,b\to\infty$.  But
$f_-^{-1}(C) \ni f_-^{-1}(c)$ which is uniformly bounded away from $c$ by
Lemma~\ref{dzetak}.  Hence $\abs{C}/\abs{G^-} \to 0$ uniformly as
$a,b\to\infty$.  The proof for the gap on the other side follows by symmetry.
Note that this argument relies heavily on the fact that the critical point stays away from the boundary throughout the family and
that the distortion of the diffeomorphic parts are uniformly bounded.

To prove that $\opR f$ is close to standard maps we need some notation.
We will concentrate on the diffeomorphic part $\tilde{\phi}_+$ of $\tilde{f}=\opR f$. The estimates for the other diffeomorphic part are similar. Decompose the branch $f_+=\phi_+\circ q$. Let 
$I=q(C^+)\subset \Dom(\phi_+)$ and $I_k=f_-^k\circ \phi_+(C^+)$. Let
$\phi=[\phi_+|I], \phi_k=[f_-|I_k]\in \text{Diff}^3([0,1])$ be the rescaled versions of the restrictions. Observe, 
$
\tilde{\phi}_+=\phi_{b-1}\circ \cdots \phi_2\circ \phi_1\circ \phi,
$
and, using the Sandwich Lemma from \cite{M2},
\begin{equation}\label{c2d}
\dist_{C^2}(\tilde{\phi},\id)=O\left( |\eta|+|D\eta| +
\sum_{k=0}^{b-1} \left\{|\eta_k|+|D\eta_k| \right\} \right),
\end{equation}
where $|\eta|$, $|D\eta|$, $|\eta_k|$, and $|\eta_k|$ are the $C^0$--norms of the nonlinearities and the derivatives of the nonlinearities of $\phi$ and $\phi_k$, respectively.

Use the Zoom Lemma from \cite{M2} and
$$
|\eta|= O(|I|), \text{ and } |D\eta|=O(|I|^2).
$$
As we saw in the  proof  for (\ref{gaps}), we can make these contributions to (\ref{c2d}) as small as needed by taking $a, b$ very large. Similarly,
$$
|\eta_k|=O(|I_k|), \text{ and } |D\eta|=O( |I_k|^2).
$$
As we observed in the first part of this proof, these intervals are all uniformly away from the critical point, for $a,b$ large. Away from the critical point the maps in the family will have uniform bounds on the nonlinearity and the derivative of the nonlinearity of $f$. From (\ref{gaps}) and the Koebe Lemma we see that all intervals $I_k$ have more and more  empty space around them when $a,b$ are taken larger and larger. We can make the sum of all lengths 
$|I_k|$ as small as needed by taking $a,b$ large enough. Hence,
$$
\dist_{C^2}(\tilde{\phi},\id)=O 
\left( |\eta|+|D\eta| +
\sum_{k=0}^{b-1} |I_k| \right)\le \epsilon
$$
when $a,b$ are large enough.
\comm{
The bound~\eqref{Qclose} follows from \eqref{gaps} and the Koebe lemma.  To
prove this, we need to show that the distortion of the diffeomorphic parts of
the renormalization shrinks to zero uniformly.  It suffices to consider the
distortion of $f^b: f_-^{-b}(C) \to C$ and $f^a: f_+^{-a}(C) \to C$.
The images of the monotone extension of these two branches cover the left and
right gaps.
Hence, \eqref{gaps} and the Koebe lemma shows that the distortion tends to
zero uniformly as $a,b\to\infty$.}
\end{proof}

The following three propositions are the tools we use to control the position
of the critical point of the consecutive renormalizations.

\begin{prop}\label{cprime}
There exists $\gamma>0$ such that for every $\epsilon>0$ and $K>0$ there exists
$N\ge 1$ such that if $a,b\ge N$,
$\frac{1}{K}\le \frac{a}{b}\le K$, and $f_{a,b}\in D_{a,b}$ is the full vertex,
then
\begin{equation}\label{newc}
    1-\epsilon
    \leq
    \frac{c'}{1-c'}
    \left(
        \frac{\hat{c}}{1-\hat{c}}
        \left(
            \gamma \frac{D\hat{f}(0)^b}{D\hat{f}(1)^a}
            \right)^{\frac 1 \alpha}
        \right)^{-1}
    \leq
    1+\epsilon,
\end{equation}
where $c'=\text{crit}(\opR f_{a,b})$, $\hat f$ is the full map and
$\hat c = \text{crit}(\hat f)$.
\end{prop}

\newcommand\fixedarr[1]{\xrightarrow{\mathmakebox[1.5em]{#1}}}
\begin{proof}
Let $f = f_{a,b}$.  Since $\opR f$ is assumed to be full it follows that all of the
following maps are onto
\begin{gather*}
    C^- = C \cap [0,c) \fixedarr{q} q(C^-) \fixedarr{\phi_-} f(C^-) \fixedarr{f^a} C
    \\
    C^+ = C \cap (c,1] \fixedarr{q} q(C^+) \fixedarr{\phi_+} f(C^+) \fixedarr{f^b} C
\end{gather*}
By the mean-value theorem there exists $x_1 \in q(C_-)$, $x_2 \in f(C^-)$, $y_1 \in
q(C^+)$, $y_2 \in f(C^+)$, such that
\begin{align*}
    D\phi_-(x_1) &= \frac{\abs{f(C^-)}}{\abs{q(C^-)}}, &
    Df^a(x_2) &= \frac{\abs{C}}{\abs{f(C^-)}}, \\
    D\phi_+(y_1) &= \frac{\abs{f(C^+)}}{\abs{q(C^+)}}, &
    Df^b(y_2) &= \frac{\abs{C}}{\abs{f(C^+)}}. \\
\end{align*}
Putting all this together we arrive at
\[
    \frac{Df^b(y_2)}{Df^a(x_2)}
    =
    \frac{D\phi_-(x_1)}{D\phi_+(y_1)}
    \frac{\abs{q(C^-)}}{\abs{q(C^+)}}
    =
    \frac{D\phi_-(x_1)}{D\phi_+(y_1)}
    \frac{u}{v}
    \left( \frac{\abs{C^-}}{\abs{C^+}}
        \frac{1-c}{c} \right)^\alpha,
\]
where $u = \phi_-^{-1}(f_-(c))$ and $v = 1 - \phi_+^{-1}(f_+(c))$ by definition.
By Lemma~\ref{gamma}, Lemma \ref{geobounds}, and the Koebe Lemma the left-hand side approaches
\[
    \frac{D\hat f^b(0)}{D\hat f^a(1)} \frac{\gamma_-}{\gamma_+}
\]
and by Lemma \ref{dzetak} and Lemma \ref{fam depend} the right-hand side approaches
\[
    \frac{D\hat\phi_-(1)}{D\hat\phi_+(0)}
    \left( \frac{\abs{C^-}}{\abs{C^+}}
        \frac{1-\hat c}{\hat c} \right)^\alpha,
\]
where $\hat\phi_-$ and $\hat\phi_+$ denote the diffeomorphic parts of $\hat f$.
Let
\[
    \gamma = \frac{\gamma_-}{\gamma_+} \frac{D\hat\phi_+(0)}{D\hat\phi_-(1)}
\]
and note that $c'/(1-c') = \abs{C^-}/\abs{C^+}$ to finish the proof.
\end{proof}

The following constant is independent of the family. It only depends on the
critical exponent $\alpha>1$. Let
\begin{equation}\label{theta}
\theta_{\alpha} = 1 + \frac{\log 2}{\log 3\alpha}>1.
\end{equation}
Denote the integer part of a real number $x$ by $\floor{x}$.

\begin{prop}\label{flippingplus} Let $\hat{f}$ be the full map in the family.
If $\hat c=\text{crit}(\hat{f})\ge \frac{2}{3}$ then there exist
$\theta> \theta_\alpha>1$, and integers $N$, $n$ and $m$ such that
for every $a\ge N$ the following holds for the full vertex $f_{a,b}\in D_{a,b}$:
\begin{enumerate}
\item If
$
b=\floor{\theta a}-n,
$
then 
$
\frac{1}{8} \le \text{crit}(\opR f_{a,b})\le \frac{1}{3}.
$

\item If
$
b=\floor{\theta a}-m,
$
then
$
\text{crit}(\opR f_{a,b})\ge \frac{2}{3}.
$
\end{enumerate}
\end{prop}


\begin{proof}
Define $\theta = \log\hat f'(1) / \log\hat f'(0)$.  We first prove that
$\theta > \theta_\alpha$.  Since $\hat c \geq 2/3$,
\[
    \theta
    =
    \frac{\log\big(D\hat\phi_+(1) \alpha / (1-\hat c)\big)}%
        {\log\big(D\hat\phi_-(0) \alpha / \hat c\big)}
    \geq
    \frac{\log\big(D\hat\phi_+(1) 3 \alpha\big)}%
        {\log\big(D\hat\phi_-(0) 3\alpha / 2\big)},
\]
where $\hat\phi_-$ and $\hat\phi_+$ are the diffeomorphic parts of $\hat f$.
The right-hand side approaches $\log(3\alpha)/\log(3\alpha/2)$ as $\hat f$
approaches standard maps.  But
\[
    \frac{\log(3\alpha)}{\log(3\alpha/2)}
    =
    \left( 1 - \frac{\log 2}{\log(3\alpha)} \right)^{-1}
    >
    1 + \frac{\log 2}{\log(3\alpha)} = \theta_\alpha,
\]
which proves that $\theta > \theta_\alpha$ for $\hat f$ close enough to
standard maps.

Statement (1) follows from Proposition~\ref{cprime} and our choice of $a$ and
$b$ as we now show.  Define $r_a$ by
$b = \floor{\theta a} - n = \theta a - r_a - n$.
Note that $0 \leq r_a < 1$, for all $a$.  The expression in parenthesis in (\ref{newc}) becomes
$$
\begin{aligned}
 \frac{\hat{c}}{1-\hat{c}}
        \left(
            \gamma \frac{D\hat{f}(0)^b}{D\hat{f}(1)^a}
            \right)^{\frac 1 \alpha}&=
             \frac{\hat c}{1-\hat c} \gamma^{\frac 1 \alpha}
        \left( \frac{D\hat f(0)^\theta}{D\hat f(1)} \right)^{\frac a \alpha}
        D\hat f(0)^{-\frac{n+r_a}{\alpha}}\\
    &= \frac{\hat c \gamma^{\frac 1 \alpha}}{1-\hat c} \lambda^{n+r_a}\\
    &= K \lambda^{n+r_a},
\end{aligned}
$$
where $\lambda = \hat f'(0)^{-1/\alpha}<1$.  Note that $\theta$ was chosen
exactly so that the dependence on $a$ is under control.  Fix $\epsilon>0$ and apply
Proposition~\ref{cprime} to get
\[
    (1-\epsilon) K \lambda^{n+1} \leq \frac{c'}{1-c'} \leq (1+\epsilon) K \lambda^n,
\]
where $c' = \text{crit}(\opR f_{a,b})$.
Now choose $n$ such that $\lambda/2 < (1+\epsilon)K\lambda^n \leq 1/2$ (which is
possible since $\{(\lambda^{k+1}/2,\lambda^k/2]\}_{k\in\mathbb{Z}}$ is a
partition of $\mathbb{R}$).  Note that $n$ does not depend on $a$.  Hence,
after fixing $n$ we are still free to choose $a$ (and consequently $b$) as
large as we like.  We get that
\begin{equation} \label{cbounds}
    \frac{1-\epsilon}{1+\epsilon} \frac{\lambda^2}{2}
    \leq
    \frac{c'}{1-c'}
    \leq
    \frac 1 2.
\end{equation}
Using the fact that $\hat c \geq 2/3$ we can estimate the lower bound by
\begin{equation} \label{lowerc}
    \frac{1-\epsilon}{1+\epsilon} \frac{\lambda^2}{2}
    =
    \frac 1 2 \frac{1-\epsilon}{1+\epsilon}
        \left(D\hat\phi_-(0) \frac{\alpha}{\hat c}\right)^{-\frac 2 \alpha}
    \geq
    \frac 1 2 \frac{1-\epsilon}{1+\epsilon}
        \left(D\hat\phi_-(0) \frac{3 \alpha}{2}\right)^{-\frac 2 \alpha}.
\end{equation}
The right-hand side approaches $g(\alpha) = (3\alpha/2)^{-2/\alpha}/2$ as
$\epsilon\to0$ and as $\hat f$ approaches standard maps.  It is easy to check
that $e^{-3/e} \leq 2 g(\alpha) < 1$ for all $\alpha>1$.  Hence, by choosing
$\epsilon$ small (which we are allowed to do by increasing $N$) and $\hat f$
close to standard maps we can ensure that
\eqref{lowerc} is larger than $1/7$ (since $e^{-3/e}/2 = 0.165\ldots$).
Finally, use that $x \mapsto x/(1-x)$ is increasing with inverse $y \mapsto
y/(1+y)$ and \eqref{cbounds} to get $1/8 \leq c' \leq 1/3$ as claimed.

Statement (2) follows by choosing $m$ such that
$(1-\epsilon) K \lambda^{m+1} \geq 2$.  Then $c'/(1-c') \geq 2$ and consequently
$c' \geq 2/3$.
\end{proof}

The following Proposition is the counterpart of Proposition \ref{flippingplus}. The proofs of both are essentially the same.

\begin{prop}\label{flippingmin}
Let $\hat{f}$ be the full map in the family.
If $\hat c=\text{crit}(\hat{f})\le \frac{1}{3}$ then there exist
$\theta> \theta_\alpha>1$, and integers $N$, $n$ and $m$ such that
for every $b\ge N$ the following holds for the full vertex $f_{a,b}\in D_{a,b}$:
\begin{enumerate}
\item If
$
a=\floor{\theta b}-n,
$
then 
$
\frac{2}{3} \le \text{crit}(\opR f_{a,b})\le \frac{7}{8}.
$

\item If
$
a=\floor{\theta b}-m,
$
then
$
\text{crit}(\opR f_{a,b})\le \frac{1}{3}.
$
\end{enumerate}
\end{prop}

\comm{
Let $\theta = \hat f'(0) / \hat f'(1)$.
Let $\lambda = \hat f'(1)^{1/\alpha}$.
Choose $n$ such that $2 \leq K \lambda^n < 2\lambda$.
Then $c'/(1-c') \leq 2 e^{3/e} < 7$.
}

\newcommand\mone{m'} 
\newcommand\mtwo{m} 
\newcommand\mthree{n'}

\section{The Construction of the Example}\label{construc}

In this section we will fix an analytic monotone family $F$ of Lorenz maps with parameter domain $\Dom(F)=[0,1]^2$. Let $\frac{1}{10}\ge \epsilon_0>0$ be small enough such that the results from  \S\ref{critpts}, Proposition \ref{flippingplus} and Proposition \ref{flippingmin}, can be applied to families consisting of maps closer than $\epsilon_0$ to the standard maps. Moreover, let $K_n\to \infty$ be the sequence given by Lemma \ref{ergodicmeasure} for $\epsilon=\epsilon_0$. 

The construction of the example will be by induction. It will produce a nested sequence of full islands.
A full island $D\subset [0,1]^2$ is said to satisfy property
$\mathcal{P}(n,\epsilon)$ if the following holds:
\begin{itemize}
\item $P_1(n)$: Every map $f\in D$ is $n$ times renormalizable of consecutive
    types $(a_k, b_k)$, $k\le n$, with $a_k,b_k\ge K_k$.

\item $P_2(n)$: $\opR^nf$ is $\epsilon_0$--close to standard maps, for all
    $f\in D$.

\item $P_3(n)$: The return interval and its adjacent gaps satisfy
$$
\frac{|C_n(f)|}{|G^\pm_n(f)|}\le \epsilon_0,\quad \forall f\in D.
$$

\item $P_4(n)$: The full map $\hat{f}\in D$ satisfies
    $\text{crit}(\opR^n\hat{f})\ge \frac23$,

\item $P_5(n, \epsilon)$: For every $f$ in some neighborhood $U\subset D$ of
    the full map $\hat f$ there exists $X_n(f)$ which is a finite collection of
    pairwise disjoint monotone preimages of $C_n(f)$ corresponding to
    $X_n(\hat{f})$, such that $|X_n(f)|\ge \frac12+\epsilon$.
\end{itemize}
Note that by abuse of notation we write $X_n(f)$ to denote both a collection of
intervals as well as the union of that collection.  Hence $\abs{X_n(f)}$
denotes the Lebesgue measure of $X_n(f)$ as a subset of $[0,1]$ and $I \in
X_n(f)$ denotes an interval in the collection.

We say that $D$ satisfies property $\mathcal{P}^+(n,\epsilon)$ if it
satisfies $\mathcal{P}(n,\epsilon)$ and
\begin{itemize}
\item $P_6^+(n)$: For every $x\in X_n(f)$ there exists $t\ge n$ such that
 $$
 \tfrac{1}{t}  \# \{ i<t| f^i(x)>\text{crit}(f)\}\ge \tfrac34.
 $$
\end{itemize}
Similarly, $D$ satisfies $\mathcal{P}^-(n,\epsilon)$ if it satisfies
$\mathcal{P}(n,\epsilon)$ and
\begin{itemize}
\item $P_6^-(n)$: For every $x\in X_n(f)$ there exists $t\ge n$ such that
 $$
 \tfrac{1}{t}  \# \{ i<t| f^i(x)<\text{crit}(f)\}\ge \tfrac34.
 $$
\end{itemize}

The following proposition is the key ingredient during the inductive
construction.

\begin{prop}\label{flipplus}
Let $D_0\subset \Dom(F)$ be a full island which satisfies
$\mathcal{P}(n,\epsilon)$ with $\epsilon\le \epsilon_0$.
Then for every $\epsilon' < \epsilon$  there exists a full island
$D\subset D_0$ which satisfies $\mathcal{P}^+(\mthree,\epsilon')$, for some
$\mthree>n$, and $X_{\mthree}(f)\subset X_n(f)$ for every $f\in D$.
Furthermore, it is possible to choose $D$ so that it satisfies
$\mathcal{P}^-(\mthree,\epsilon')$
instead of
$\mathcal{P}^+(\mthree,\epsilon')$.
\end{prop}

\begin{proof} The proof has four parts. The first part is a preparation.
Note that the $n^\text{th}$ level probability basis measures behave differently
depending on whether $n$ is even or odd.

Assume first that $n\ge 1$ is odd. Consider the family $F_0: D_0\ni f\mapsto \opR^nf$. Apply Lemma \ref{fam depend},  Lemma \ref{geobounds} and Proposition \ref{flippingplus}(1) to obtain $a,b$ large enough such that the full island $D_{a,b}\subset U_0\subset D_0$ satisfies
$P_1(n+1)$, $P_2(n+1)$, $P_3(n+1)$, and $\text{crit}(\opR^{n+1}\hat{f}_{a,b})\le \frac13$, where $\hat{f}_{a,b}\in D_{a,b}$ is the full map. 

Apply Lemma \ref{C} to obtain a finite collection of monotone preimages of $C_{n+1}(\hat{f}_{a,b})$, the intervals are taken from  $\mathcal{C}_{n+1}(\hat{f}_{a,b})$, with the following properties. First denote the union by $X_{n+1}(\hat{f}_{a,b})$. Then $X_{n+1}(\hat{f}_{a,b})\subset X_n(\hat{f}_{a,b})$, and 
$$
|X_{n+1}(\hat{f}_{a,b})|> \frac{1}{2}+\epsilon'+\frac{\epsilon-\epsilon'}{2}.
$$
Lemma \ref{correspond} allows us to  choose a small enough neighborhood $U_{a,b}$ of $\hat{f}_{a,b}\in D_{a,b}\subset D_0$ such that $P_5(n+1, \epsilon'+(\epsilon-\epsilon')/2)$ also holds for every $f\in U_{a,b}$. The words describing the combinatorics of the intervals in $X_{n+1}(f)$, $f\in U_{a,b}$ are the same as the word of the corresponding interval in
$X_{n+1}(\hat{f}_{a,b})\subset X_n(\hat{f}_{a,b})$. This implies that $X_{n+1}(f)\subset X_n(f)$, for every $f\in U_{a,b}$.

Similarly, in the case when $n\ge 2$ is even we can apply Proposition \ref{flippingplus}(1) and Proposition \ref{flippingmin}(2) to turn the condition $P_4(n)$ into  $\text{crit}(\opR^n\hat{f}_{a,b})\le \frac13$.

Hence, we may assume that $\mone>n$ is {\it even} and $P_1(\mone)$,
$P_2(\mone)$, $P_3(\mone)$,
$\text{crit}(\opR^n\hat{f}_{a,b})\le \frac13$, and
$P_5(\mone, \epsilon'+(\epsilon-\epsilon')/2)$.  From now on we will consider
even renormalization levels.

\medskip

In the second part we show that by going to a sufficiently deep level we can
control the ratio of the return times of the first-return map.
We will alternate between applying Proposition \ref{flippingmin}(1) and Proposition \ref{flippingplus}(1) and then repeat the process. This will give rise to a nested sequence of full islands
$D_0\supset D_1\supset D_2\supset \cdots$.  Each time we will choose $a, b$ large enough such that each island $D_k$ satisfies $P_1(\mone+k)$, $P_2(\mone+k)$, $P_3(\mone+k)$. This is possible because of Lemma \ref{fam depend} and Lemma \ref{geobounds}.

Recall, in general
$$
\frac{T^-_{k+1}}{T^+_{k+1}}=\frac{a_{k+1}T^+_{k}+T^-_{k}}{b_{k+1}T^-_{k}+T^+_{k}}.
$$
For very large choices $a, b$ when applying Proposition \ref{flippingmin}(1) we have
$$
\frac{T^-_{2k+1}}{T^+_{2k+1}}\approx \frac{a_{2k+1}}{b_{2k+1}} \cdot 
\frac{T^+_{2k}}{T^-_{2k}}\approx \theta \cdot 
\frac{T^+_{2k}}{T^-_{2k}}\ge 
\theta_\alpha \cdot \frac{T^+_{2k}}{T^-_{2k}},
$$
where $\theta_\alpha>1$.
Similarly, for large choices of $a, b$ when applying Proposition \ref{flippingplus}(1) we have
$$
\frac{T^+_{2k+2}}{T^-_{2k+2}}\ge 
\theta_\alpha \cdot \frac{T^-_{2k+1}}{T^+_{2k+1}}.
$$
Hence, the repeated pairwise application of Proposition \ref{flippingmin}(1) and Proposition \ref{flippingplus}(1) implies, when $a,b$ are chosen large enough each time, that
$$
\frac{T^+_{2k+2}}{T^-_{2k+2}}\ge 
\theta^2_\alpha \cdot \frac{T^+_{2k}}{T^-_{2k}}.
$$
Choose $\mtwo=\mone+2k$ large enough such that
\begin{equation}\label{Tratio}
\frac{T^+_{\mtwo}}{T^-_{\mtwo}}\ge \kappa\gg1,
\end{equation}
where $\kappa$ is a very large number to be determined later.  Note that $\mtwo$ is
even.

\medskip

The third part of the proof constructs a large set of points given by $P_5$
which later will be shown to behave according to $P_6^+$.  Let
$\hat{f}_{\mtwo}\in D_{\mtwo}$ be the full map. Observe,
\begin{equation}\label{cpos}
\text{crit}(\opR^{\mtwo}\hat{f}_{\mtwo})\in \left[\tfrac18, \tfrac13\right].
\end{equation}

Consider the collection $\mathcal{C}_{\mtwo}(\hat{f}_{\mtwo})$ of monotone preimages of $C_{\mtwo}(\hat{f}_{\mtwo})$. From Lemma \ref{C} we get 
$|\mathcal{C}_{\mtwo}(\hat{f}_{\mtwo})|=1$ and $\mathcal{C}_{\mtwo}(\hat{f}_{\mtwo})$ is a refinement of $\mathcal{C}_n(\hat{f}_{\mtwo})$ since $\mtwo>n$. Choose a finite subcollection  $X_{\mtwo}(\hat{f}_{\mtwo})\subset 
\mathcal{C}_{\mtwo}(\hat{f}_{\mtwo})$ contained in $X_n(\hat{f}_{\mtwo})\subset \mathcal{C}_n(\hat{f}_{\mtwo})$ such that
$$
|X_{\mtwo}(\hat{f}_{\mtwo})|\ge \frac{1}{2}+\epsilon'+\frac{\epsilon-\epsilon'}{4}. 
$$
For every connected component $I\subset X_{\mtwo}(\hat{f}_{\mtwo})$ there is $e_I\ge 0$ such that
$$
\hat{f}_{\mtwo}^{e_I}: I\to C_{\mtwo}(\hat{f}_{\mtwo})
$$
is monotone and onto. Let $E=\max \{e_I\}$.

For $f\in D_{\mtwo}$ close enough to $\hat{f}_{\mtwo}$ let $X_{\mtwo}(f)$ be the corresponding collection of monotone preimages of $C_{\mtwo}(f)$. Here we used Lemma \ref{correspond}. If the neighborhood $U^{(0)}\subset D_{\mtwo}$ of $\hat{f}_{\mtwo}$  from which $f$ is chosen is small enough then
$$
|X_{\mtwo}(f)|\ge \frac{1}{2}+\epsilon'+\frac{\epsilon-\epsilon'}{8}, 
$$
for every $f\in U^{(0)}$.

Let $f\in U^{(0)}$ and $Z\subset C_{\mtwo}(f)$. The pullback of $Z$ is the set
$$
P_f(Z)=\bigcup_{I \in X_{\mtwo}(f)} f^{-e_I}(Z)\cap I \subset X_{\mtwo}(f).
$$
Then there exists $r>0$, which only depends on $\epsilon_0$,  such that if $|Z|/|C_{\mtwo}(f)|\ge 1-r$ then
\begin{equation}\label{PZbound}
|P_f(Z)|\ge  \frac{1}{2}+\epsilon'+\frac{\epsilon-\epsilon'}{16}.
\end{equation}
The reason is that all the branches $f^{e_I}:I\to C_{\mtwo}(f)$ have a uniformly bounded distortion as a consequence of $P_3(\mtwo)$ and the Koebe Lemma~\cite{MS}.

The renormalization of $\hat{f}_{\mtwo}$ is a full map closer than $\epsilon_0$ to the standard full map and has a critical point in a controlled  interval, see (\ref{cpos}). From Lemma \ref{measurebalance} we get a lower bound $w>0$ on the weight of the (ergodic) absolutely continuous invariant measure on both sides of the critical point. 

Let $T\ge 1$ and consider
$$
Z_T=\left\{ x\in C_{\mtwo}(\hat{f}_{\mtwo}) \,\middle|\,
\tfrac{1}{T}\#\{i<T\mid (\opR^{\mtwo}_0\hat{f}_{\mtwo})^i(x)>\text{crit}(\hat{f}_{\mtwo})\}\ge \tfrac12 w \right\},
$$
where $\opR^{\mtwo}_0\hat{f}_{\mtwo}:C_{\mtwo}(\hat{f}_{\mtwo})\to C_{\mtwo}(\hat{f}_{\mtwo})$ is the non-rescaled version of the $\mtwo^\text{th}$ renormalization. The Ergodic Theorem allows us to choose an arbitrarily large 
$$
T_0\ge \kappa E
$$  
such that
$$
\frac{|Z_{T_0}|}{|C_{\mtwo}(\hat{f}_{\mtwo})|}\ge 1-\frac12 r.
$$
The large constant $\kappa$ will be chosen later.
Observe, $Z_{T_0}$ is a finite collection of monotone onto branches $I$ of 
$(\opR^{\mtwo}_0\hat{f}_{\mtwo})^{T_0}: I\to C_{\mtwo}(\hat{f}_{\mtwo})$. For maps in a small enough neighborhood $U^{(1)}\subset U^{(0)}\subset D_{\mtwo}$,  let $Z_{T_0}(f)$ be the corresponding collection of monotone preimage under $\opR^{\mtwo}_0f$.  If the neighbor $U^{(1)}$ is small enough then 
$$
\frac{|Z_{T_0}(f)|}{|C_{\mtwo}(f)|}\ge 1-\frac34 r.
$$
Observe, if $I\subset Z_{T_0}(f)$ is a connected component then
$$
\tfrac{1}{T_0}\#\{i<T_0| (\opR^{\mtwo}_0f)^i(I)>\text{crit}(f)\}\ge \tfrac12 w.
$$

Let $\mthree=\mtwo+1$ and note that $\mthree$ is odd. Consider a full island $D_{a,b}\subset U^{(1)}\subset D_{\mtwo}$ of the family $F_{\mtwo}: D_{\mtwo}\ni f\mapsto \opR^{\mtwo}f$.

For $f\in D_{a,b}$ let $\mathcal{M}_{a,b}(f)\subset 
\mathcal{C}_{\mtwo}(f)$ be all monotone preimages $I$ of $C_{\mthree}(f)$ which are contained in $ C_{\mtwo}(f)$, and with transfer time $e_I\ge T_0$.
From Lemma \ref{geobounds}, (\ref{gaps}),  we get that $|C_{\mthree}(f)|\to 0$ when $a,b\to \infty$. Moreover, there are only finitely many monotone preimages with transfer time smaller than $T_0$. Hence,
\begin{equation}\label{measZTab}
\frac{|\mathcal{M}_{a,b}(f)|}{|C_{\mtwo}(f)|}\to 1,
\end{equation}
when $f\in D_{a,b}$ and $a,b\to \infty$. Let 
$$
Z_{T_0, a,b}(f)=\{ I \in \mathcal{M}_{a,b}(f)| I\subset Z_{T_0}(f)\}.
$$
From (\ref{measZTab}) we get
\begin{equation}\label{zbound}
\frac{|Z_{T_0, a,b}(f)|}{|C_{\mtwo}(f)|}\ge  1-\frac{7}{8} r,
\end{equation}
when $f\in D_{a,b}$ and $a,b$ large enough.  

Define $D_{\mthree}=D_{a,b}\subset D_{\mtwo}$ by choosing $a,b$ according to Proposition \ref{flippingmin}(1) but also large enough to have the above estimate (\ref{measZTab}) and the result from Lemma \ref{geobounds} for $\epsilon=\epsilon_0$. The full island $D_{\mthree}$ satisfies
$P_1(\mthree)$, $P_2(\mthree)$, $P_3(\mthree)$, and $P_4(\mthree)$.

For a map
$f\in D_{\mthree}$ let
$$
X_{\mthree}(f)=P_f(Z_{T_0, a,b}(f))\subset X_{\mtwo}(f).
$$
From the estimates (\ref{PZbound}), and (\ref{zbound}) we obtain that this set
satisfies $P_5(\mthree,\epsilon')$. Moreover,
\begin{equation}\label{contained}
X_{\mthree}(f)\subset X_{\mtwo}(f)\subset X_n(f),
\end{equation}
a property  which is part of the proposition.

\medskip

So far we did not yet discuss the statistical behavior of the maps, properties
$P_6^\pm$.
The fourth part will show $P_6^+(\mthree)$. Take $x\in X_{\mthree}(f)$. There exists $E\ge e(x)\ge 1$ with
$$
f^{e(x)}(x)\in C_{\mtwo}(f),
$$
and
$$
\tfrac{1}{T_0}  \#\{i<T_0\mid (\opR^{\mtwo}_0f)^i(f^{e(x)}(x))> \text{crit}(f)\}\ge \tfrac12 w.
$$
Let
$$
\#^+=\#\{i<T_0\mid (\opR^{\mtwo}_0f)^i(f^{e(x)}(x))> \text{crit}(f)\}
$$
and 
$$
\#^-=\#\{i<T_0\mid (\opR^{\mtwo}_0f)^i(f^{e(x)}(x))< \text{crit}(f)\}.
$$
Observe, 
$$
\frac{\#^+}{T_0}\ge \frac12 w,
$$
and, using $\#^++\#^-=T_0$,
$$
\frac{\#^-}{\#^+}\le \frac{2}{w}-1.
$$
Let
$$
t=\#^+\cdot T^+_{\mtwo}+\#^-\cdot T^-_{\mtwo}+e(x).
$$
Let $C_0^+(f)=[\text{crit}(f),1]$ and $C_0^-(f)=[0,\text{crit}(f)]$. Then 
$$
\begin{aligned}
\tfrac{1}{t} \#\{i<t&\mid f^i(x)> \text{crit}(f)\}\\
&\ge
\frac{\#^-\cdot  T^-_{\mtwo}\cdot \mu_{\mtwo}^-(C_0^+(f))+\#^+\cdot T^+_{\mtwo}\cdot \mu_{\mtwo}^+(C_0^+(f))}{\#^+\cdot T^+_{\mtwo}+\#^-\cdot T^-_{\mtwo}+e(x)}\\
&\ge
\frac{\mu^+_{\mtwo}(C_0^+(f))}
{ 1+ \frac{\#^-}{\#^+}\cdot \frac{T^-_{\mtwo}}{T^+_{\mtwo}}+\frac{E}{\#^+ \cdot T^+_{\mtwo}}}\\
&\ge
\frac{1-\epsilon_0}
{1+(\frac{2}{w}-1)\cdot \frac{1}{\kappa}+\frac{2}{ w} \cdot \frac{1}{\kappa} }.
\end{aligned}
$$
Recall, $\epsilon_0<\frac{1}{10}$. Hence, by choosing $\kappa\ge 1$ large enough we can assure
that
$$
\tfrac{1}{t} \#\{i<t\mid f^i(x)> \text{crit}(f)\}\ge \tfrac34.
$$

Finally, note that by looking at odd renormalization levels we could have
chosen the ratios $T_m^+/T_m^-$ small instead of large in the second part which
would lead to $P_6^-(\mthree)$ instead of $P_6^+(\mthree)$ in the fourth part.
In other words, by suitably modifying the above argument we can choose $D$
so that it satisfies $\mathcal{P}^-(\mthree,\epsilon')$ instead of
$\mathcal{P}^+(\mthree,\epsilon')$.
\end{proof}

\begin{thm}\label{example} Every analytic monotone family of Lorenz maps has an infinitely renormalizable map $f$ which  satisfies
\begin{itemize}
\item $f$ is ergodic with respect to Lebesgue measure,
\item $\OO_f$ is the attractor,
\item $|\OO_f|=0$,
\item $\OO_f$ carries exactly two ergodic measures,
\item $f$ has no physical measure.
\end{itemize}
\end{thm}

\begin{proof}  Let $\epsilon_N=\frac1N\epsilon_0$. Without loss of generality we may assume that $D\hat{f}(0), D\hat{f}(1)>1$, where $\hat{f}$ is the full map in the family. Apply Theorem \ref{monotonefull}, Lemma \ref{geobounds} and Proposition \ref{cprime}. This gives a full island $D_1$ which satisfies the properties $P_1(1)$, $P_2(1)$, $P_3(1)$, and $P_4(1)$. 
Let $\hat{f}_1$ be the full map in $D_1$ and $U_1\subset D_1$ a small neighborhood of $\hat{f}_1$. We can choose a set $X_1(\hat{f}_1)$ which is a finite collection of pairwise disjoint monotone preimages  of $C_1(\hat{f})$ such that
$|X_1(\hat{f}_1)|\ge \frac12+2\epsilon_1$. Then for maps $f\in U_1\subset D_1$ we let $X_1(f)$ be the corresponding collection of monotone preimages of $C_1(f)$. If the neighborhood $U_1\subset D_1$ is chosen small enough, all the maps in $U_1$ will also satisfy $P_5(1,\epsilon_1)$.  Hence $D_1$ satisfies $\mathcal P(1,\epsilon_1)$.

By repeatedly applying Proposition~\ref{flipplus} we construct a sequence of
full islands $D_1\supset D_2\supset D_3 \supset  \cdots$ such that for $N\geq2$
$$
D_N \text{ satisfies } \mathcal{P}^+(m_N,\epsilon_N) \text{ for } N \text{ even,}
$$
and 
$$
D_N \text{ satisfies } \mathcal{P}^-(m_N,\epsilon_N)  \text{ for } N \text{ odd,}
$$
with $m_N\to \infty$. For example, given $D_{2N}$ which satisfies
$\mathcal{P}^+(m_{2N},\epsilon_{2N})$, apply Proposition \ref{flipplus} with $\epsilon=\epsilon_{2N+1}$ to obtain $D_{2N+1}$ which satisfies $\mathcal{P}^-(m_{2N+1},\epsilon_{2N+1})$, etc.

 Let 
$$
f\in \bigcap D_N.
$$
From Lemma \ref{ergnowand} and property $P_3(m_N)$ we know that $f$ is ergodic with attractor the Cantor set $\OO_f$, and the attractor has measure zero. From Lemma \ref{ergodicmeasure} we know that $\OO_f$ carries exactly two ergodic measures. Let 
$$
X(f)=\bigcap X_{m_N}(f).
$$
The construction implies $|X(f)|\ge \frac12$. Moreover, for every $x\in X(f)$ we have
$$
\limsup \frac{1}{t} \# \{i< t| f^i(x)>\text{crit}(f)\}\ge \frac34,
$$
and
$$
\liminf \frac{1}{t} \# \{i< t| f^i(x)>\text{crit}(f)\}\le \frac14.
$$
Hence, $f$ does not have a physical measure.
\end{proof}

\begin{rem} The construction used to prove Theorem \ref{example} leaves enough freedom. In particular, one can construct a Cantor set of maps in the given family which all satisfy the properties mentioned in the theorem.
\end{rem}


\begin{thebibliography}{CMMT}










\bibitem[CMMT]{CMMT} V.V.M.S. Chandramouli, M. Martens, W. de Melo, C.P. Tresser,  Chaotic period doubling, Erg. Th. and Dyn. Sys. 29, (2009), 381-418.


\bibitem[CT]{CT} P. Coullet, C. Tresser.
It\'eration d'endomorphismes et groupe de renormalisation. J.
Phys. Colloque C 539, C5-25 (1978).


\bibitem[F]{F} M.J. Feigenbaum.
Quantitative universality for a class of non-linear
 transformations. J. Stat. Phys.,  19 (1978), 25-52.


\bibitem[GM]{GM} J.-M. Gambaudo, M. Martens, Algebraic Topology for Minimal Cantor Sets, Ann. Henri Poincar\'e, 7(3), (2006), 423-446.





\bibitem[HK]{HK} F. Hofbauer, G. Keller, Quadratic maps without asymptotic measure, Comm. Math. Phys. 127, (1990), 319-337.

\bibitem[J]{J} S. Johnson, Singular measures without restrictive
 intervals, Comm.  Math. Phys. 110, (1987), 185-190.





\bibitem[L]{Lo} E.N. Lorenz, Deterministic non-periodic flow, J. Atmos. Sci. 20, (1963), 130-141.


 




\bibitem [M1]{M1} M. Martens, Distortion Results and Invariant Cantor sets
for Unimodal maps, Erg.Th and Dyn.Sys. 14, (1994), 331-349.

\bibitem [M2]{M2} M. Martens, The Periodic Points of Renormalization,
Ann. of Math. 147, (1998), 543-584.

\bibitem[MM]{MM} M. Martens, W.\ de Melo, Universal models for Lorenz maps, Ergodic Theory Dynam. Systems (2001), 21(3), 833-860.

\bibitem[MW]{MW} M. Martens, B. Winckler, On the Hyperbolicity of Lorenz Renormalization,   Comm. Math. Phys. (2013),  325(1),  185-257. 





\bibitem[MS]{MS} W. de Melo, S. van Strien, One-dimensional dynamics, 
Springer Verlag, Berlin, 1993.




\bibitem[T]{T} W. Tucker, The Lorenz attractor exists, C. R. Acad. Sci. Paris S\'er I Math. 328.12, (1999), 1197-1202.

\bibitem[V]{V} M. Viana, What's new on Lorenz strange attractors?, Math.\ Intelligencer, 22(3), (2000), 6-19

\bibitem[W]{W} B. Winckler, Renormalization of Lorenz Maps, PhD Thesis KTH, Stockholm, Sweden, 2011.


\end{thebibliography}
\end{document}